\documentclass[a4paper,openany, 12pt]{article}

\usepackage{titlesec}
\titleformat*{\section}{\centering\bfseries\Large}
\titleformat*{\subsection}{\centering\bfseries\large}

\usepackage{booktabs}
\usepackage[T1]{fontenc}
\usepackage[utf8]{inputenc}
\usepackage[english]{babel}
\usepackage[a4paper,top=4cm,bottom=3cm,left=2.5cm,right=2.5cm]{geometry}
\usepackage{amsfonts}
\usepackage{mathtools}

\usepackage{lmodern}
\usepackage{empheq}
\usepackage{amsthm}
\usepackage{amssymb}
\usepackage{booktabs}
\usepackage{caption}
\usepackage{siunitx}
\usepackage{amssymb}
\usepackage{graphicx}
\usepackage{stmaryrd}
\usepackage{color}
\usepackage{pifont}
\usepackage{enumerate}
\usepackage{enumitem}
\setlist[description]{leftmargin=\parindent,labelindent=\parindent}
\usepackage{listings} 
\usepackage{amsmath}
\usepackage{amscd}
\usepackage{braket}
\usepackage{physics}
\usepackage{tikz}
\usepackage{calc}
\newcommand*{\avint}{\mathop{\ooalign{$\int$\cr$-$}}}
\newcommand*{\haus}{\mathcal{H}}

\newtheorem{theorem}{Theorem}[section]
\newtheorem{corollary}[theorem]{Corollary}
\newtheorem{lemma}[theorem]{Lemma}

\usepackage{titlesec}
\usepackage{lipsum}

\titleformat{\section}
{\normalfont\scshape}{\thesection}{1em}{}

\title{Generalization of Klain's Theorem to Minkowski Symmetrization of compact sets and related topics}
\author{Jacopo Ulivelli }
\date{}
\begin{document}
	
\maketitle

\begin{abstract}
	We shall prove a convergence result relative to sequences of Minkowski symmetrals of general compact sets. In particular, we investigate the case when this process is induced by sequences of subspaces whose elements belong to a finite family, following the path marked by  Klain in \cite{6}, and the generalizations in \cite{3} and \cite{11}. 
	We prove an analogous result for fiber symmetrization of a specific class of compact sets.
	The idempotency for symmetrizations of this family of sets is investigated, leading to a simple generalization of a result from Klartag \cite{25} regarding the approximation of a ball through a finite number of symmetrizations, and generalizing an approximation result in \cite{9}\footnote{\textit{Email}: jacopo.ulivelli@gmail.com}
\end{abstract}

\section{Introduction}
The tool of symmetrization has played an important role in Mathematics since its very introduction from Steiner of the homonym \textit{Steiner Symmetrization} in the attempt of proving the Isoperimetric Inequality. One of the main features of this tool is that there exist a sequence of hyperplanes such that the corresponding sequence of successive symmetrizations ensure the convergence to a ball independently from the compact convex body we start from. Moreover, Steiner symmetrization preserves the volume of the object we are modifying. Historically, this symmetrization  is employed in standard proofs of not only the isoperimetric inequality, but also of other powerful geometric inequalities, such as the Brunn-Minkowski,  Blaschke -Santalò  or Petty projection inequalities.\\To this day, even though this tool is almost two hundred years old, it still plays a crucial role in mathematical research. For example it has been recently employed by Milman and Yehudayoff \cite{5} in the solution of a long-standing open problem about Affine Quermassintegrals.

The interest toward the convergence of sequences of successive symmetrizations has risen again in the last years thanks to a series of papers focusing on Steiner symmetrization (for example Klain \cite{6},  Bianchi, Burchard, Gronchi and Volcic \cite{11}, Bianchi, Klain, Lutwak, Yang and Zhang \cite{4}, Volcic \cite{19} and the very recent Asad and Burchard \cite{20}) and on Minkowski symmetrization (as Klartag \cite{25}-\cite{28} and Coupier and  Davydov \cite{23}).\\
In Bianchi, Gardner and Gronchi \cite{10}-\cite{3} the authors introduced a wider frame for the study of general symmetrizations, studying the common features and the important properties of these different tools. In particular in \cite{3} they provide a beautiful generalization of Klain's main result in \cite{6} valid for Steiner Symmetrization to many other symmetrizations, including Minkowski's.\\Still, many of these works focus on the behaviour of symmetrization in the class of compact convex sets.

The aim of this paper will be the study of some convergence results in the class of compact sets, mainly on the frame of Minkowski symmetrization and the similar fiber symmetrization. Indeed, as usual the case of compact sets reveals many pathologic issues and interesting developments. Here we will try to present some results in this direction. In particular there are some properties, such as idempotency and invariance on symmetric sets, that we lose once turning to the general compact case. Many of them rely on the behaviour of the Minkowski addition, with which we define Minkowski symmetrization, thus pointing our focus to the study of Minkowski addition of compact set. See  Fradelizi,  Madiman and Marsiglietti \cite{8} for an extensive survey on the subject.

Let us introduce some terminology. Let $\mathcal{E}$ be the class $\mathcal{K}^n_n$ of compact convex bodies with nonempty interior in $\mathbb{R}^n$,  or the class $\mathcal{C}^n$ of compact sets in $\mathbb{R}^n$. Given a subspace $H\subset\mathbb{R}^n$  let $\diamondsuit_H$ denote a symmetrization over $\mathcal{E}$, i.e. any map which associates to every set in $\mathcal{E}$ a set in $\mathcal{E}$ symmetric with respect to $H$. Given a sequence  $\{H_m\}$ of subspaces and $K \in \mathcal{E}$ we consider the sequence of sets\[ K_m=\diamondsuit_{H_m}\dots\diamondsuit_{H_2}\diamondsuit_{H_1} K.\]
For which sequences  $\{H_m\}$ and for which symmetrizations $\diamondsuit_H$ does the sequence $\{K_m\}$ converge for each $K\in\mathcal{E}$? This process depends  on the class $\mathcal{E}$, on the definition of $\diamondsuit_H$ and on the sequence  $\{H_m\}$ (and, in particular, on the dimension of the subspaces). 

The cases which have been studied most are those when $\diamondsuit$ is Steiner, Schwarz or Minkowski symmetrization in  the class $\mathcal{E}=\mathcal{K}_n^n$ and $\{H_m\}_{m \in \mathbb{N}}$ consists of hyperplanes. Some results are available also for more general symmetrizations, for the class of compact sets and for subspaces of any dimension (see \cite{11}, \cite{3}). In this notation, our aim is to extend some of these results to the class $\mathcal{E}=\mathcal{C}^n$.

We will start from the study of Minkowski symmetrization of compact sets. In particular, we will prove a generalization of the following result due to Klain \cite{6}. Here $S_H$ denotes the Steiner symmetrization with respect to a subspace $H$.
\begin{theorem}[Klain]\label{T_1}
	Given $K \in \mathcal{K}_n^n$ and a finite family $\mathcal{F}=\{Q_1,...,Q_l\} \subset \mathcal{G}(n,n-1)$, consider a sequence of subspaces $\{H_m\}_{m \in \mathbb{N}}$ such that for every $m \in \mathbb{N}$, $H_m=Q_j$ for some $1\leq j \leq l$. Then the sequence \[K_m:=S_{H_m}...S_{H_1} K\] converges to a body $L \in \mathcal{K}_n^n$. Moreover, $L$ is symmetric with respect to $Q_j$ for every $Q_j$ which appears infinitely often in the sequence.
\end{theorem}
This result has been extended in \cite{3} for the class $\mathcal{K}_n^n$ to Minkowski symmetrization, to fiber symmetrization and also to general symmetrizations satisfying certain properties (the proper definitions will be introduced in the next section). Regarding Steiner symmetrization this result was generalised in \cite{11} to the class $\mathcal{C}^n$.

Our strategy will be to reduce our argument to the convex case so that, starting from the general compact case, we will appropriately compare and link the behaviour of the two different processes. In particular we prove the following.
\begin{theorem}\label{TM}
	Consider $K \in \mathcal{K}^n$ and a sequence of isometries $\{\mathbb{A}_m\}_{m \in \mathbb{N}}$. If the sequence \[K_m=\frac{1}{m} \sum_{j=1}^m \mathbb{A}_j K \] converges, then the same happens for every compact set $C \in \mathcal{C}^n$ such that conv$(C)=K$. Moreover, the two sequences converge to the same limit.
\end{theorem}
The strength of this result reveals itself once that one realises that the Minkowski symmetrization of a set $C \in\mathcal{C}^n$ can be expressed as the mean of two isometries of $C$, i.e. the identity and the reflection with respect to the considered symmetrization subspace. Iterations of this symmetrization can be expressed in an analogous fashion. A similar result has been proven in Theorem 7.4 in \cite{3}, where this property of Minkowski symmetrization was first observed; this result states that if a sequence of subspaces is such that the corresponding sequence of symmetrals converges to a ball for every convex compact set we start from, then the same will happen starting from an arbitrary compact set. Theorem \ref{TM} can be seen as a more general version of such result, because we do not require to converge to a ball and in general it allows us to work with specific sequences too.

This approach is not  strong enough when working with fiber symmetrization. Indeed, as we will point out later, the symmetrization behaves way less predictably and the corresponding sequences do not necessarily converge to convex sets. For this symmetrization we prove convergence only for compact sets that satisfy \[\partial \text{conv} C \subseteq C,\] where $\partial \text{conv} C$ is the boundary of the convex envelope of $C$, providing the following result. For a fixed subspace $H$ we define the corresponding fiber symmetrization of a compact set as the union of the Minkowski symmetrizations of its orthogonal sections with respect to $H$. See equation \eqref{fiber} and the corresponding definition for more details.
\begin{theorem}\label{TKF}
	Let $K \in \mathcal{C}^n$ such that  $\partial \text{conv}(K) \subset K$, let $\mathcal{F}=\{Q_1,...,Q_s\}$ be a family of subspaces of $\mathbb{R}^n$, $1 \leq\text{dim}(Q_i)\leq n-2$, and let $\{H_m\}$ be a sequence such that $H_m \in \mathcal{F}$ for every $m \in \mathbb{N}$. Then the sequence \[K_m:=F_{H_m}...F_{H_1} K\] converges to a convex set $L$, where $L$ is the limit of the same symmetrization process applied to conv$(K)$. Thus $L$ is symmetric with respect to all the subspaces of $\mathcal{F}$ which appear infinitely often in $\{H_m\}$.
\end{theorem}
The case of 1-dimensional sections needs to be treated differently, and in a further work we will prove this result for $n \geq 1$ with the additional hypothesis $\abs{K}>0$, i.e. $K \in \mathcal{C}_n^n$.\\The peculiarity of the case $n=1$ will be clearer after stating Theorem \ref{Tb}, which is instrumental to prove Theorem \ref{TKF}.

The study of this issues will lead us to focus on the following problem: what can we deduce on the Minkowski sum of two compact sets from their boundaries? We will see how this question is strongly related to the idempotency of Minkowski and fiber symmetrizations.\\On this regards we obtain the following result.
\begin{theorem}\label{Tb}
	Let $K,L$ be compact sets with connected boundary such that, for every $x \in \mathbb{R}^n$, neither $K+x$ is strictly contained in $-L$ nor $-L$ is strictly contained in $K+x$. Then, \[K+L=\partial K + \partial L.\]
\end{theorem}
A similar result was proven recently from  Fradelizi, Làngi, Zvavitch \cite{9} in a more restrictive case. Notice that this does not give us informations about the case $n=1$, which we will address in Lemma \ref{lemmone} and turns out to give us interesting properties on the finite Minkowski addition of bounded sets in $\mathbb{R}$.

The author would like to thank the Referee for the constructive comments and recommendations which helped to improve the readability and quality of the present paper.

\section{Preliminaries}

As usual, $S^{n-1}$ denotes the unit sphere in the Euclidean $n$-space $\mathbb{R}^n$ with Euclidean norm $\norm{\cdot}$. The term $ball$ will always mean an $n$-dimensional Euclidean ball, and the unit ball in $\mathbb{R}^n$ will be denoted by $B^n$. $B(x,r)$ is the open ball with center $x$ and radius $r$. If $x,y \in \mathbb{R}^n$, we write $x\cdot y$ for the inner product. If $x \in \mathbb{R}^n \setminus \{o\}$, then $x^\perp$ is the $(n-1)$-dimensional subspace orthogonal to $x$. $\mathcal{G}(n,i)$ denotes the Grassmanian of the $i$-dimensional subspaces of $\mathbb{R}^n$, $1\leq i \leq n-1$, and if $H \in \mathcal{G}(n,i)$, $H^\perp$ is the $(n-i)$-dimensional subspace orthogonal to $H$. By \textit{subspace} we mean \textit{linear subspace}. Given $x \in \mathbb{R}$, $\lfloor x \rfloor$ is the floor function of $x$.

If $X$ is a set, we denote by conv$X$ its convex envelope, and by $\partial X$ its boundary. If $H \in \mathcal{G}(n,i)$, then $X|H$ is the (orthogonal) projection of $X$ on $H$. If $X$ and $Y$ are sets in $\mathbb{R}^n$ and $t\geq 0$, then $tX:=\{tx:x \in X\}$ and \[X+Y:=\{x+y: x \in X, y \in Y \} \] denotes the \textit{Minkowski sum} of $X$ and $Y$. For $X$ measurable set its volume will be $\abs{X}$. 

When $H \in \mathcal{G}(n,i),$ we write $R_H$ for the \textit{reflection} with respect to $H$, i.e. the map that takes $x\in \mathbb{R}^n$ to $2(x|H)-x$, where $x|H$ is the projection of $x$ onto $H$. If $R_H X=X$, we say that $X$ is $H$-symmetric.

We denote by $\mathcal{C}^n$ the class of nonempty compact subsets of $\mathbb{R}^n$. $\mathcal{K}^n$ will be the class of non empty compact convex subsets of $\mathbb{R}^n$ and $\mathcal{K}_n^n$ is the class of \textit{convex bodies}. A body is a compact subset of $\mathbb{R}^n$ equal to the closure of its interior. In the analogous way we define $\mathcal{C}_n^n$. If $K \in \mathcal{K}^n$, then \[h_K(x):=\sup\{x\cdot y: y \in K\},\] for $x \in \mathbb{R}^n$, defines the \textit{support function} $h_K$ of $K$. With the support function we can define the mean width of a convex body $K$, which is \[w(K):=\frac{1}{\abs{\partial B^n}} \int_{S^{n-1}} (h_K(\nu)+h_K(-\nu)) d\nu,\]
where $\abs{\partial B^n}$ is the $(n-1)$-dimensional measure of the unit sphere in $\mathbb{R}^n$.
The aforementioned spaces $\mathcal{C}^n$ and $\mathcal{K}^n$ are metric spaces with the \textit{Hausdorff metric}, which is given in general for two sets $A,B$ by \[d_\haus(A,B):=\sup\{e(A,B),e(B,A)\}, \] where \[e(A,B):=\sup_{x\in B}d(x,A)\] is the \textit{excess} of the set $A$ from the set $B$, and $d(x,A)$ is the usual distance between a point and a set.
The completeness of such metric spaces is a classic result \cite{17}, we will refer to it as \textit{Blaschke selection Theorem} both in convex and compact context.

Another classical result we will refer to is the Brunn-Minkowski inequality. Given two measurable sets $A,B$ such that $A+B$ is measurable (the sum of measurable set is not always measurable), it states that \[\abs{A+B}^{1/n}\geq \abs{A}^{1/n}+\abs{B}^{1/n},\] where equality holds if and only if one of the following holds: either $A$ is convex with non empty interior and $B$ is a homothetic copy of $A$ (up to removing subsets of volume zero) or both $A$ and $B$ have null measure and lie on parallel hyperplanes. See Gardner \cite{27} for a precise and comprehensive survey on the general case of this inequality.

Given $C \in \mathcal{C}^n, H \in \mathcal{G}(n,i),1\leq i \leq n-1$, we recall the definition of some symmetrizations:
\begin{itemize}
	\item{Schwarz symmetrization:} \[S_H C:=\bigcup_{x \in H} B(x,r_x),\] where $r_x$ is such that for the $(n-i)$-dimensional measure of the sections holds $\abs{B(x,r_x)}=\abs{C \cap (H^\perp+x)}$, and $B(x,r_x) \subset H^\perp +x$. If $\abs{C \cap (H^\perp+x)}=0$ then $r_x=0$ when $C \cap (H^\perp+x) \neq \emptyset$, while when the section is empty then the symmetrization keeps it empty. Observe that thanks to Fubini's Theorem Schwarz symmetrization preserves the volume. 
	
	For $i=n-1$ we have \textit{Steiner symmetrization}.
	
	\item{Minkowski symmetrization:}\[M_H C:=\frac{1}{2}(C+R_H C).\] We will also consider the case $i=0$, which is called the central Minkowski symmetrization \[M_o K=\frac{K-K}{2}.\] 
	
	Notice that for support functions we have the following: \[h_{K}+h_L=h_{K+L} \qquad  \forall K,L \in \mathcal{K}^n,\] thus Minkowski symmetrization preserves mean width.
	
	\item{Fiber symmetrization:} 
	\begin{equation}\label{fiber}
		F_H C:= \bigcup_{x \in H} \left[\frac{1}{2}\left(C\cap(H^\perp+x)\right)+\frac{1}{2}\left(R_H C \cap (H^\perp+x)\right) \right].
	\end{equation}
	Observe that, defining $M_{H^\perp,x}$ the central Minkowski symmetrization with respect to $x$ in $H^\perp+x$ identified with $\mathbb{R}^{n-i}$, we can write \[F_H K=\bigcup_{x \in H} M_{H^\perp,x} (K \cap (H^\perp+x)).\]
	
	\item{Minkowski-Blaschke symmetrization:} If $K$ is a convex compact set we define
	\begin{equation*} 
		h_{\overline{M}_H K}(u):=
		\begin{cases}
			\avint \limits_{S^{n-1}\cap (H^\perp + u)} h_K (v) dv, \quad \text{if } \abs{S^{n-1}\cap (H^\perp + u)}\neq 0 \text{ in } \mathbb{R}^{n-i}\\
			h_K(u), \qquad \text{otherwise.}
		\end{cases}
	\end{equation*}
	for every $u \in S^{n-1}$.
	At the end of Section 3 we will see that we can extend this definition to any compact set using the support function of its convex envelope.
\end{itemize}

Consider a family of bodies $\mathcal{E}$ and a subspace $H \in \mathcal{G}(n,i)$, then an $i$-symmetrization is a map \[\diamondsuit_H:\mathcal{E} \to \mathcal{E}_H ,\] where $\mathcal{E}_H$ are the $H$-symmetric elements of $\mathcal{E}$.

We now state some properties of $i$-symmetrizations. As we will see, the interaction between these properties and eventually their lack can be crucial in the study of convergence. Consider $K, L \in \mathcal{E}$, $H$ a subspace in $\mathbb{R}^n$, the map $\diamondsuit_H$  satisfies the properties of:
\begin{description}
	\item{\textit{Monotonicity}} if $K \subset L \Rightarrow \diamondsuit_H K \subset \diamondsuit_H L$;
	\item{\textit{H-symmetric invariance}} if $R_H K=K \Rightarrow \diamondsuit_H K= K$;
	\item{\textit{Invariance under translations orthogonal to H of H-symmetric sets}} if $R_H K=K, y \in H^\perp \Rightarrow \diamondsuit_H(K+y)=\diamondsuit_H K$.
\end{description}

Notice that these properties hold for Steiner, Minkowski and fiber symmetrizations, while the first and the third hold for Schwarz symmetrization.

\section{Klain's Theorem for Minkowski symmetrization of compact sets}

Two of the main features of Steiner, Schwarz, Minkowski and fiber symmetrizations are the idempotency and the invariance for $H$-symmetric bodies in the class of convex sets. These two properties no longer hold when we switch to the class of generic compact sets.

An immediate example regarding Minkowski symmetrization is the following. Consider in $\mathbb{R}^2$ the compact set $C=\{(-1,0),(1,0)\}$. This set is obviously symmetric with respect to the vertical axis, which we can identify with a subspace $H$. Then we have \[M_H C=\{(-1,0),(0,0),(1,0)\},\] thus the invariance for symmetric sets no longer holds. If we apply again the same symmetrization, \[M_H(M_H C)=\{(-1,0),(-1/2,0),(0,0),(1/2,0),(1,0)\},\] showing that the same happens to idempotency. In Figures \ref{F_2} and \ref{F_3} we see an example for the fiber symmetrization of a compact set in the plane. 

\begin{figure}
	\centering
	
	\begin{tikzpicture}[scale=0.5]
		
		\draw[-latex] (-3.5,0)--(3.5,0) node [below]{$V^\perp$};
		\draw[-latex] (0,-3.5)--(0,3.5) node [left]{$V$};
		\draw[fill=red,opacity=0.3] (-3,3)--(-3,-3)--(0,0)--(3,-3)--(3,3) node [above,black] {$K$}--(0,0)--(-3,3);
		\draw (3,3) node [above,black] {$K$};

		\draw[-latex] (4.5,0)--(11.5,0) node [below]{$V^\perp$};
		\draw[-latex] (8,-3.5)--(8,3.5) node [left]{$V$};
		\draw[fill=red,opacity=0.3] (5,3)--(5,-3)--(7,-1)--(8,-3)--(9,-1)--(11,-3)--(11,3) node [above, black] {$F_V K$}--(9,1)--(8,3)--(7,1)--(5,3);
		\draw (11,3) node [above, black] {$F_V K$};
	\end{tikzpicture}
	\caption{}\label{F_2}
\end{figure}
\begin{figure}
	\centering
	\begin{tikzpicture}[scale=0.5]
		
		\draw[-latex] (-3.5,0)--(3.5,0) node [below]{$V^\perp$};
		\draw[-latex] (0,-3.5)--(0,3.5) node [left]{$V$};
		\draw[fill=red,opacity=0.3] (-3,3)--(-3,-3)--(-1.8,-1.8)--(-1.5,-3)--(-0.6,-1.8)--(0,-3)--(0.6,-1.8)--(1.5,-3)--(1.8,-1.8)--(3,-3)--(3,3) node [above, black] {$F_V^2 K$}--(1.8,1.8)--(1.5,3)--(0.6,1.8)--(0,3)--(-0.6,1.8)--(-1.5,3)--(-1.8,1.8)--(-3,3);
		\draw (3,3) node [above, black] {$F_V^2 K$};
		
		\draw[-latex] (4.5,0)--(11.5,0) node [below]{$V^\perp$};
		\draw[-latex] (8,-3.5)--(8,3.5) node [left]{$V$};
		\draw[fill=red,opacity=0.3] (5,3)--(5,-3)--(5.6666,-2.3333)--(5.75,-3)--(6.3333,-2.3333)--(6.5,-3)--(7,-2.3333)--(7.25,-3)--(7.6666,-2.3333)--(8,-3)--(8.3333,-2.3333)--(8.75,-3)--(9,-2.3333)--(9.5,-3)--(9.6666,-2.3333)--(10.25,-3)--(10.3333,-2.3333)--(11,-3)--(11,3) --(10.3333,2.3333)--(10.25,3)--(9.6666,2.3333)--(9.5,3)--(9,2.3333)--(8.75,3)--(8.3333,2.3333)--(8,3)--(7.6666,2.3333)--(7.25,3)--(7,2.3333)--(6.5,3)--(6.3333,2.3333)--(5.75,3)--(5.6666,2.3333)--(5,3);
		\draw (11,3) node [above, black] {$F_V^3 K$};
	\end{tikzpicture}
	\caption{}\label{F_3}
\end{figure}

If we iterate this process for $C=\{(-1,0),(1,0)\}$, we see that in this case there is no finite degree of idempotency, i.e. there does not exist an index $\ell \in \mathbb{N}$ such that \[M_H^\ell C=M_H^{k+\ell} C\] for every $k \in \mathbb{N}$, where in general \[\underbrace{M_H \dots M_H}_{\ell\text{-times}}:=M_H^\ell.\] Moreover the iterated symmetrals converge to the set given by conv$(C)$. This is the main idea behind the next result, after proving a technical Lemma.
\begin{lemma}\label{Lm}
	Let $K \in \mathcal{C}^n$, $H$ a subspace of $\mathbb{R}^n$. Then
	
	i) for every $v \in \mathbb{R}^n$ \[M_H(K+v)=M_H(K)+v|H,\]
	
	ii) if $K$ is $H$-symmetric, then $K\subseteq M_H K$,
	
	iii) $K=M_H K$ if and only if $K$ is convex and $H$-symmetric.
\end{lemma}
\begin{proof}
	The first statement follows from the explicit calculations \[M_H(K+v)=\frac{K+v+R_H(K+v)}{2}=\frac{K+R_H(K)}{2}+\]\[\frac{v|H^\perp+v|H-v|H^\perp+v|H}{2}=M_H(K)+v|H,\] where we used the linearity of the reflections and the decomposition $v=v|H+v|H^\perp$. 
	
	For the second statement, by hypothesis we have that $R_H K=K$, i.e. $R_H(x) \in K$ for every $x \in K$. Then, taking $x \in K$, $(x+R_H(R_H(x)))/2=x \in M_H K$, concluding the proof. 
	
	Consider now $K$ such that $K=M_H K$. Then obviously $K$ must be $H$-symmetric, and $K=R_H K$. Then, for every $x,y \in K$ we have that $(x+y)/2 \in K$, thus for every point $z$ in the segment $[x,y]$ we can build a sequence by bisection such that it converges to $z$. $K$ is compact, henceforth it contains $z$. The other implication is trivial.
\end{proof}
If we consider the iterated symmetral 
\begin{equation}\label{D1}
	K_m:=M_H^m K=\underbrace{M_H...M_H}_\text{m-times} K,
\end{equation}
we notice that the second statement implies that $K_m \subseteq K_{m+1}$ for every $m \in \mathbb{N}$. 

The next Theorem presents the intuition behind the rest of the work. It will actually be an easy corollary of the results we will prove later, but we present it because its proof is selfcontained and it is useful to glimpse the underlying structure of Minkowski symmetrization.

\begin{theorem} \label{T_2}
	Let $K \in \mathcal{C}^n$, $H \in \mathcal{G}(n,i), 1 \leq i \leq n-1$. Then the sequence $K_m$ in \eqref{D1} converges in Hausdorff distance to the $H$-symmetric convex compact set \[L=\text{conv}(M_H K).\]
\end{theorem}

\begin{proof}
	We observe preliminarily that for the properties of convex envelope and Minkowski sum we have  $K_m \subseteq L$ for every $m \in \mathbb{N}$. Then we only need to prove that for every $x \in L$ we can find a sequence $x_m \in K_m$ such that $x_m \to x$. We can represent $K$ as $\bar{K}+v, v \in K$, where $\bar{K}$ contains the origin. Since Minkowski symmetrization is invariant under $H$-orthogonal translations, we can take $v \in H$. 
	
	For every $m$ we have $R_H K_m=K_m$, and thus we can write \[K_{m+1}=M_H K_m=\frac{K_m+K_m}{2}=\frac{\overbrace{{K_1+...+K_1}}^{2^m-\text{times}}}{2^{m}}.\] Considering the aforementioned representation of $K$, $R_H K=R_H \bar{K} + v$, and we have \[K_{m}=\bar{K}_m+v, \quad \text{ where } \bar{K}_m:=M_H^m \bar{K},\] thus we can write every point $y \in K_m$ as $y=\bar{y}+v, \bar{y} \in \bar{K}_m$.
	
	Given $x\in L$, thanks to Carathéodory Theorem there exist $x_k \in K_1, \lambda_k \in (0,1), k=1,...,n+1$ such that $\sum_{k=1}^{n+1} \lambda_i=1$ and \[x=\sum_{k=1}^{n+1} \lambda_k x_k=\sum_{k=1}^{n+1} \lambda_k \bar{x}_k + v,\] where $x_k=\bar{x}_k+v, \bar{x}_k \in \bar{K}_1$ . For every $\lambda_k$ we consider its binary representation \[\lambda_k=\sum_{\ell=1}^{+\infty} \frac{a_{\ell,k}}{2^\ell}, \qquad a_{\ell,k} \in\{0,1\}\]  (we do not consider $\ell=0$ because $\lambda_i < 1$), and its $m$-th approximation given by the partial sum \[\lambda_{m,k}:=\sum_{\ell=1}^{m} \frac{a_{\ell,k}}{2^\ell}=\frac{1}{2^m}\sum_{\ell=1}^m a_{\ell,k} 2^{m-\ell}.\] We notice for later use that $\abs{\lambda_k-\lambda_{m,k}}\leq 1/2^m$.
	
	Calling $q_s:=\lfloor 2^{s}/(n+1) \rfloor$ we now build the sequence \[x_s:=\sum_{k=1}^{n+1} \lambda_{q_s,k} \bar{x}_k+v = \frac{1}{2^{q_s}} \sum_{k=1}^{n+1} \left(\sum_{\ell=1}^{q_s} a_{\ell,k} 2^{q_s-\ell}\right)\bar{x}_k +v,\] where the $2^{s}- q_s(n+1)$ spare terms in $\bar{K}_1$ can be taken as the origin in the sum representing $\bar{K_s}$.
	
	Then we have that every $x_s$ belongs to $K_s$, and \[\norm{x-x_s}=\norm{\bar{x}+v-(\bar{x}_s+v)} \leq \sum_{k=1}^{n+1} \norm{\bar{x}_k} \abs{\lambda_k-\lambda_{q_s,k}}  \leq \frac{1}{2^{q_s}} \sum_{k=1}^{n+1} \norm{\bar{x}_k}\leq (n+1) \frac{\max_{y \in K_1} \norm{y-v}}{2^{q_s}}.\] Clearly $\norm{x-x_s} \to 0$ as $s \to +\infty$, which concludes our proof. 
\end{proof}

As an immediate consequence we have the following result.

\begin{corollary} \label{C_1}
	Let $K \in \mathcal{C}^n$, $H \in \mathcal{G}(n,i), 1 \leq i \leq n-1$. Then we have that the sequence \[K_m:=F_H^m K=\underbrace{F_H...F_H}_\text{m-times} K\] converges in Hausdorff distance to the $H$-symmetric compact set \[L=\bigcup_{x \in H} \text{conv}(F_H K \cap (x+H^\perp)).\]
\end{corollary}
\begin{proof}
	Recalling the definition of fiber symmetrization \[F_H K=\bigcup_{x \in H} \frac{1}{2}((K\cap(x+H^\perp))+(R_H K \cap (x+H^\perp)))=\bigcup_{x \in H} M_{H^\perp,x}(K\cap (x+H^\perp)).\] The result is a straightforward application of Theorem \ref{T_2} to the sections of $K$.
\end{proof}

\paragraph{Remark.} In Corollary \ref{C_1} we lose the convexity on the limit, but convexity of its sections still holds, as a consequence of Theorem \ref{T_2}. This property is known, when dim$(H)=n-1$, as \textit{directional convexity} (see \cite{21}). We can extend this concept to \textit{sectional convexity}, that is, fixed a subspace $H$ in $\mathbb{R}^n$ and a set $A$, the convexity of every section $A \cap (x+H^\perp), x \in H$. Then in the previous result the sectional convexity is with respect to the subspace $H$.

We now state Shapley-Folkman-Starr Theorem (\cite{24},\cite{7} Theorem 3.1.2) for using it in the next proof.

\begin{theorem}\label{sfs}[Shapley-Folkman-Starr]
	Let $A_1,...,A_k \in \mathcal{C}^n$. Then \[d_\haus(\sum_{j=1}^k A_j,\text{conv}(\sum_{j=1}^k A_j ) )\leq \sqrt{n} \max_{1\leq j \leq k} D(A_j),\] where $D(\cdot)$ is the diameter function $D(K):=\sup\{\norm{x-y}: x,y \in K\}$.
\end{theorem}

In Theorem \ref{T_2} we already saw how the convexification effect of Minkowski addition works when we iterate the same symmetrization. Now, with the inequality given by Shapley-Folkman-Starr Theorem, we generalise the former result proving Theorem \ref{TM}.

\begin{proof}[Proof of Theorem \ref{TM}]
	First notice that orthogonal transformations and Minkowski addition commute with convex envelope. Thus for $C_m=\sum_{j=1}^m \mathbb{A}_j C/m$, where $C \in\mathcal{C}^n$ and $\{\mathbb{A}_j\}$ is a sequence of isometries as in the hypothesis,  \[\text{conv}(C_m)= \text{conv} \left( \frac{1}{m}\sum_{j=1}^m \mathbb{A}_j C \right)= \frac{1}{m} \sum_{j=1}^m \mathbb{A}_j \text{conv} (C)= \frac{1}{m} \sum_{j=1}^m \mathbb{A}_j K= K_m.\]
	
	We now apply Shapley-Folkman-Starr Theorem, obtaining \[d_\haus (C_m,K_m)=d_\haus(C_m, \text{conv}C_m) \leq \frac{\sqrt{n}}{m}\max_{1\leq j \leq m} D(\mathbb{A}_j C)=\frac{\sqrt{n}}{m}\max_{1\leq j \leq m} D(C).\]
	$C$ is compact and thus bounded, hence $d_{\haus}(C_m,K_m)\to 0$, completing the proof. In fact the compactness is not necessary and boundedness would suffice, but this is beyond the interest of the present paper.
\end{proof}

\begin{corollary}\label{T_3}
	Let $K$ be a convex compact set and let $\{H_m\}$ be  a sequence of subspaces of $\mathbb{R}^n$ (not necessarily of the same dimension) such that the sequence of iterated symmetrals \[K_m:=M_{H_m}...M_{H_1} K\] converges to a convex compact set $L$ in Hausdorff distance. Then the same happens for every compact set $\tilde{K}$ such that $conv(\tilde{K})=K$, and the sequence $\tilde{K}_m$, defined as $\tilde{K}_m:= M_{H_m}\dots M_{H_1} \tilde{K}$, converges to the same limit $L$.
\end{corollary}
\begin{proof}
	
	We will show that the theorem holds proving that \[d_\haus(\tilde{K}_m, K_m) \to 0\] for $m \to \infty$. 
	
	We can write $K_m$ as the mean of Minkowski sum of composition of reflections of $K$. Indeed, we have 
	\begin{equation*}
		\begin{gathered}
			K_1=\frac{K+R_{H_1} K}{2},\\
			K_2=\frac{K+R_{H_1} K + R_{H_2}(K+R_{H_1} K)}{4}=\frac{K+R_{H_1}K+R_{H_2} K + R_{H_2} R_{H_1} K}{4},\\
			...
		\end{gathered}
	\end{equation*}
	and so on. The same obviously holds for $\tilde{K}_m$. We call these compositions of reflections $\mathbb{A}_j, 1\leq j \leq 2^m$, and defining $A_j:=\mathbb{A}_j \tilde{K}$ we can write \[\tilde{K}_m=\frac{1}{2^m}\sum_{j=1}^{2^m} \mathbb{A}_j \tilde{K}=\frac{1}{2^m}\sum_{j=1}^{2^m} A_j.\]
	
	The proof follows applying Theorem \ref{TM}.
\end{proof}

We now have, as a consequence of Corollary \ref{T_3}, our generalization of Klain's result.

\begin{corollary} \label{MK}
	Let $K \in \mathcal{C}^n$, $\mathcal{F}=\{Q_1,...,Q_s\} \subset \mathcal{G}(n,i), 1\leq i \leq n-1$, $\{H_m\}$ a sequence of elements of $\mathcal{F}$. Then the sequence \[K_m:=M_{H_m}...M_{H_1}K\] converges to a convex set $L$ such that it is the limit of the same symmetrization process applied to $\bar{K}=conv(K)$. Moreover, $L$ is symmetric with respect to all the subspaces of $\mathcal{F}$ which appear infinitely often in $\{H_m\}$.
\end{corollary}
\begin{proof}
	The proof follows straightforward from the generalization of Klain Theorem to Minkowski symmetrization of convex sets (cf. \cite{3}, Theorem 5.7) and Corollary \ref{T_3}.
\end{proof}

We can use a similar method to generalise the following classical result from Hadwiger, see for example \cite{7}, Theorem 3.3.5.
\begin{theorem}\label{TH}[Hadwiger]
	For each convex body $K \in \mathcal{K}_n^n$ there is a sequence of rotation means of $K$ converging to a ball.
\end{theorem}

Then the next result is obtained combining Theorems \ref{TH} and \ref{TM}.

\begin{corollary}
	For each compact set $C$ such that conv$(C) \in \mathcal{K}_n^n$ there is a sequence of means of isometries $C$ converging to a ball.
\end{corollary}

\paragraph{Remark.} Corollary \ref{T_3} gives us an answer regarding the possibility of extending the Minkowski-Blaschke symmetrization $\overline{M}_H$ to compact sets. This symmetrization, that we have defined in Section 2 for convex bodies, can be practically seen as the mean of rotations of a compact set $K \in \mathcal{K}^n$ by a subgroup of $SO(n)$, thus can be approximated by \[\frac{1}{N}\sum_{k=1}^N\mathbb{A}_k K,\] where $\{\mathbb{A}_k\}_{k=1}^N \subset \{\mathbb{A}_k\}_{k \in \mathbb{N}}$ is a suitable set of rotations dense in said subgroup.

Indeed, from the definition of $\overline{M}_H$ in terms of the support function, we have that the integral can be approximated by \[\sum_{k=1}^N \frac{h_K(\mathbb{A}_k^{*} x)}{N}=\frac{1}{N}\sum_{k=1}^N h_{\mathbb{A}_k K} (x),\] which corresponds naturally to the Minkowski sum written above. 

Then again, following the proof of Corollary \ref{T_3}, we can write the symmetral as the limit of a mean of Minkowski sum of isometries of a fixed $K \in \mathcal{K}^n$, and thus Minkowski-Blaschke symmetrization actually gives the same result for every $C \in \mathcal{C}^n$ such that conv$(C)=K$.\\
This shows that this symmetrization is sensible only to the extremal points of a set, thus it makes no difference in using it with compact sets or convex sets.

\section{The case of Convex outer boundary}

One of the main properties of Minkowski symmetrization is that, as a consequence of the Brunn-Minkowski inequality, it increases the volume. Indeed, for every measurable set $K \subset \mathbb{R}^n$ such that $\abs{K}>0$ and $M_H K$ is measurable, we have \[\abs{M_H K}^{1/n}=\abs{1/2(K+R_H K)}^{1/n}\geq \frac{1}{2}\abs{K}^{1/n}+\frac{1}{2}\abs{R_H K}^{1/n}=\abs{K}^{1/n},\] where equality holds if and only if  $K$ and $R_H K$ are homothetic convex bodies from which sets of measure zero have been removed. See \cite{27} for a general proof of this result. We work only with compact sets, thus the equality condition is possible only if the two bodies are homothetic and convex. This happens if and only if $K=M_H K$, thus we would like to state that  the iteration of Minkowski symmetrization increases the volume until the sequence of symmetrals reaches $M_H \text{conv} (K)$. 

With Theorem \ref{T_2} we proved that, regardless of the volume, the limit of $\tilde{K}_m$ is actually $M_H \text{conv}(K)$, but now we raise one more question: can we obtain this limit in a finite number of iterations? Under which hypothesis is this possible? 

We start by giving an answer for compact sets of $\mathbb{R}$. This case is more complicated than for similar objects in $\mathbb{R}^n, n \geq 2$, as we will prove later.
\begin{lemma}\label{lemmone}
	Let $K \in \mathbb{R}$ be a compact set such that $\text{conv}(K)=[a,b]$ with the following property:  \[\exists \varepsilon>0 \text{  s.t.  } [a,a+\varepsilon] \subset K  \text{  or  } [b-\varepsilon,b]\subset K.\]
	
	Then there exists an index $\ell \in \mathbb{N}$ depending on $\varepsilon$ and $(b-a)$ such that \[M_o^{\ell} K= M_o^{\ell+k} K\] for every $k\in \mathbb{N}$. 
	
	Moreover, $\ell$ increases with $(b-a)$ and decreases if $\varepsilon$ increases.
\end{lemma}

\begin{proof}
	First consider the case $K \supseteq \{a\}\cup[b-\varepsilon,b]$. Then \[M_o K \supseteq M_o (\{a\}\cup[b-\varepsilon,b])\supset \left[\frac{a-b}{2},\frac{a-b}{2}+\frac{\varepsilon}{2}\right] \cup \left[\frac{b-a}{2}-\frac{\varepsilon}{2},\frac{b-a}{2}\right].\] Easy calculations show that the same happens when $K \supseteq [a,a+\varepsilon]\cup \{b\}$. Then, naming \[M:=\frac{b-a}{2}, \quad m:=\frac{b-a}{2}-\frac{\varepsilon}{2},\] and we can work with a set containing a subset the form \[ [-M,-m] \cup [m,M]=:\tilde{K}, \] where $M-m=\varepsilon/2$.
	
	If now we apply the symmetrization, we obtain 
	\begin{equation}\label{k0}
		M_o K\supseteq [-M,-m]\cup \left[ \frac{m-M}{2},\frac{M-m}{2}\right]\cup [m,M]=M_o \tilde{K}.
	\end{equation}
	If $(M-m)/2\geq m$, that is $m \leq M/3$, then $M_o K=\text{conv}(K)$, and the result holds with $\ell=1$.
	
	In the general case we can show by induction that the following inclusion holds
	\[ M_o^{k+1} K \supseteq M_o^{k+1} \tilde{K} \supseteq \bigcup_{j=0}^{2^{k+1}} \left[ \frac{(2^{k+1}-j)m-jM}{2^{k+1}},\frac{(2^{k+1} -j)M-jm}{2^{k+1}} \right],\] where the first inclusion is trivial thanks to the monotonicity of Minkowski symmetrization. In particular we will show that \[M_o^{k+1}\tilde{K}\supseteq M_o^{k}\tilde{K}\cup \bigcup_{j=1}^{2^{k}} \left[\frac{(2^{k+1}-2j+1)m-(2j-1)M}{2^{k+1}}, \frac{(2^{k+1}-2j+1)M-(2j-1)m}{2^{k+1}} \right], \] which is the desired set. This inclusion is actually an equality, but proving this fact is beyond our goal here.
	
	For $k=0$ we have already seen in (\ref{k0}) that the inclusion holds. By inductive hypothesis, at the $(k+1)$-th step the means of adjacent intervals of $M_o^k \tilde{K}$ are given by
	\begin{equation*}
		\begin{gathered}
			\frac{1}{2}\left[\frac{(2^k-(j+1))m-(j+1)M}{2^k},\frac{(2^k-(j+1))M-(j+1)m}{2^k}\right]+ \\ \frac{1}{2}\left[ \frac{(2^k-j)m-jM}{2^k},\frac{(2^k -j)M-jm}{2^k} \right] \\= \left[\frac{(2^{k+1}-2(j+1)+1)m-(2(j+1)-1)M}{2^{k+1}},\frac{(2^{k+1}-2(j+1)+1)M-(2(j+1)-1)m}{2^{k+1}} \right]
		\end{gathered}
	\end{equation*}
	for every $j=0,...,2^{k}-1$, giving us the elements of the union with odd indices.
	
	Observe that $M_o^{k}\tilde{K}$ is invariant under reflection. Thus, thanks to Lemma \ref{Lm} and the monotonicity of Minkowski symmetrization, we have $M_o^{k}\tilde{K} \subseteq M_o^{k+1} K$, concluding the induction. 
	
	Taking at the $k$-th step two adjacent intervals, we have that they are connected if \[\frac{(2^k-(j+1))M-(j+1)m}{2^k}\geq \frac{(2^k-j)m-jM}{2^k}.\]
	It follows that the condition for filling the whole segment conv$(M_H^k K)$ is \[\frac{m}{M}\leq\frac{2^k-1}{2^k+1}.\] Observe that the dependence on the index $j$ disappeared after calculations, confirming that this holds for every couple of adjacent intervals.
	
	By hypothesis $M-m=\varepsilon/2$ and $(2^k-1)/(2^k+1) \to 1 $. We have \[\frac{m}{M}=1+\frac{m-M}{M}=1-\frac{\varepsilon}{2M}, \] then there exists $\ell \in \mathbb{N}$ such that \[1-\frac{\varepsilon}{2M} < \frac{2^\ell-1}{2^\ell+1},\] thus $M_o^\ell K= \text{conv}(K)$ for\[\ell \geq \log_2 \left(\frac{4M}{\varepsilon}-1 \right).\] This set is convex and $o$-symmetric, thus is invariant under Minkowski symmetrization. The dependence from $M$ and $\varepsilon$ is clear from the last inequality.
\end{proof}

Notice that it is crucial that either $a$ or $b$ belong to an interval with positive measure contained in $K$. Indeed, if the two extremes both were isolated points there would occur a situation analogous to the example presented at the beginning of Section 3, thus there would be a part of the set which stabilises itself only at the limit. 

\paragraph{Remark.} With wider generality the previous Lemma holds for the means of Minkowski sums. Indeed, if $K \subset \mathbb{R}$, for every $x \in \mathbb{R}$ holds \[\frac{1}{m}\sum_{j=1}^m(K-x)=\frac{1}{m}\sum_{j=1}^m K -x,\] and taking $x$ as the mean point of the extreme points of $K$ we reduce ourself to the same context of the Lemma, which can be restated as follows.

\begin{lemma}\label{L_3}
	Let $K \in \mathbb{R}$ be a compact set such that $\text{conv}(K)=[a,b]$ with the following property:  \[\exists \varepsilon>0 \text{  s.t.  } [a,a+\varepsilon] \cup [b-\varepsilon,b]\subset K.\]
	
	Then there exist an index $\ell \in \mathbb{N}$ depending on $\varepsilon$ and $(b-a)$ such that \[\frac{1}{2^\ell}\sum_{j=1}^{2^\ell} K=\frac{1}{2^{\ell+k}}\sum_{j=1}^{2^{k+\ell}} K\] for every $k\in \mathbb{N}$.
	
	Moreover, $\ell$ increases with $(b-a)$ and decreases if $\varepsilon$ increases.
\end{lemma}
\begin{proof} First we remind the reader that, as we have seen in Theorem \ref{T_2}, when we iterate $M_H$, after the first symmetrization we are just computing the mean \[\frac{1}{2^{m-1}}\sum_{j=1}^{2^{m-1}} M_H K = M_H^{m} K.\]
	Moreover, we observe that the only difference with the previous Lemma is that we do not have the sum with the reflection, so we have to require in the hypothesis that both the end-points of $K$ belong to segments included in $K$.\\
	Now we can work with a set \[\tilde{K}:=([-M,-m] \cup [m,M] )+ x\] for a suitable $x \in \mathbb{R}$, and the rest of the proof follows as the previous one.
\end{proof}

A weaker property that these sets have is to contain the boundary of their convex envelope. When $n \geq 2$, this is enough to prove the stronger and more general result in Corollary \ref{cb}.

\begin{lemma}\label{lb}
	Let $K,L \in \mathcal{C}^n$ such that $\partial K$, $\partial L$ are connected and $K \cap L \neq \emptyset$. If neither $L$ is strictly contained in $K$ nor $K$ is strictly contained in $L$, then there exists $z \in \partial K \cap \partial L$.
\end{lemma}
\begin{proof}
	First note that in general if $K$ is a closed set and $\partial K$ is connected then $K$ is connected. Moreover, $\mathbb{R}^n \setminus \text{int}K$ is connected too. 
	
	Observe that if $K=L$ then $\partial K \cap \partial L = \partial K = \partial L \neq \emptyset$ and there would be nothing to prove, so we will work in the hypothesis $K \neq L$.
	
	We start proving that $\partial K \cap L \neq \emptyset$. Indeed, there exist $y \in L \setminus K$ and $x \in K \cap L$. Then, since $L$ is connected, there exists a continuous curve $\gamma$ joining $x,y$. Now, $\gamma$ must cross $\partial K \cap L$ going from one end ($x$, inside $K$) to the other ($y$, outside $K$) in a point $u$ which belongs to the required intersection.
	
	Now we prove that $\partial K \setminus L \neq \emptyset$. Indeed, there exists $x \in K \setminus L$ and $K,L$ are compact, then there exists $r>0$ such that the ball $B(0,r)$ contains strictly  $K$ and $L$. Then, there exists a continuous curve $\gamma'$ from $x$ to the boundary of $B(0,r)$ that does not intersect $\partial L$ because of the connectedness of $\mathbb{R}^n \setminus \text{int} L$. Moreover $\gamma'$ must cross $\partial K$ in a point $v$ that does not belong to $L$, hence this point belongs to $\partial K \setminus L$.
	
	Finally, since $\partial K$ is connected, we can join $u,v$ with a curve contained in $\partial K$ from inside $L$ to outside of it, crossing $\partial L$ in at least one point $z \in \partial K \cap \partial L$.
\end{proof}

If $A$ is a connected compact set, then we call the \textit{external connected component} of $\mathbb{R}^n\setminus A$ the unbounded connected component of such a set. Then we notice that, as in \cite{9}, this result holds also for the boundary of the external connected component of $\mathbb{R}^n \setminus K$ and $\mathbb{R}^n \setminus L$. Moreover we point out that the hypothesis of Lemma \ref{lb} immediately rule out the case $n=1$. This will indeed be an issue in Corollary \ref{cb} and Theorem \ref{TKF}.

Now we can prove Theorem \ref{Tb}.
\begin{proof}[Proof of Theorem \ref{Tb}]
	Let $x \in K+L$, then there exist $\kappa \in K, \ell \in L$ such that $x=\kappa+\ell$. If we define $\tilde{K}:=K+x-\kappa, \tilde{L}:=-L+x+\ell$, we have that $x \in \tilde{K} \cap \tilde{L}$ hence $\tilde{K}$ and $\tilde{L}$ satisfy the hypothesis of Lemma \ref{lb}. Thus $\partial \tilde{K} \cap \partial \tilde{L} \neq \emptyset$. 
	
	Let $z \in \partial \tilde{K} \cap \partial \tilde{L}$, then \[z-x+\kappa \in \partial K, \ell-z+x \in \partial L.\] Now \[(z-x+\kappa)+(\ell-z+x)=\kappa+\ell,\] proving our assertion.
\end{proof}

\begin{corollary}\label{cb}
	Let $K \in \mathcal{K}^n$ and let $H$ be a subspace of $\mathbb{R}^n$. Then 
	\begin{equation}\label{ecb}
		M_H K=M_H \partial K,
	\end{equation}
	In particular, if $C \in \mathcal{C}^n$ and $C \supseteq \partial \text{conv}(C)$, then $M_H C$  is convex, and \[M_H \text{conv}(C)=M_H C.\]
	
	The same holds for fiber symmetrization if $H$ is not a hyperplane.
\end{corollary}
\begin{proof}
	We first prove the result regarding Minkowski symmetrization. We apply Theorem \ref{Tb} to $K/2$ and $R_H K/2$. Indeed, observe that the two sets are convex, hence with connected boundary. Moreover, since they have the same volume, no translate of one set is strictly contained in the other set. Then, Theorem \ref{Tb} yields $M_H K=M_H \partial K$.
	
	Consider now a set $C \in \mathcal{C}^n$ with $\partial \text{conv}(C) \subseteq C$. From equation (\ref{ecb}), $\partial \text{conv}(C) \subseteq \partial C$ and the monotonicity of Minkowski symmetrization, we infer \[M_H C \supset \frac{\partial C + \partial R_H C}{2} \supseteq \frac{\partial \text{conv}(C)+ \partial R_H \text{conv}(C)}{2} = M_H \text{conv}(C).\] Since the reverse inclusion trivial, this concludes the proof in the case of Minkowski symmetrization.
	
	Regarding fiber symmetrization notice that if $H$ was a hyperplane then the sections are one-dimensional and in Lemma \ref{lemmone} we proved that we need certain conditions on the boundary to obtain idempotency. 
	In general we know that fiber symmetrization preserves convexity, thus $F_H \text{conv} C$ is convex, and its boundary is given by the union of the boundaries of the sections by $H^\perp + x, x \in H$. If $H$ is not a hyperplane, these sections are obtained by Minkowski symmetrization of convex sets of dimension greater or equal than two, completing the proof.
\end{proof}

The proof of Theorem \ref{TKF} now follows immediately.

\begin{proof}[Proof of Theorem \ref{TKF}]
	By Corollary \ref{cb} we have $F_{H_1} K=F_{H_1} \text{conv}K$. Therefore $F_{H_1} K \in \mathcal{K}_n^n$, and it suffices to apply to for the rest of the sequence the generalization of Klain's Theorem for fiber Symmetrization (cf. \cite{3}, Theorem 5.6), proving the Theorem.
\end{proof}

We conclude this section with another immediate application, a small addition to Klartag's following result; cf. Theorem 1.1 in \cite{25}. The same generalization holds for similar results in \cite{28}.
\begin{theorem}
	Let $n \geq 2, 0 < \epsilon< 1/2$, and let $K \subset \mathbb{R}^n $ be a compact set such that $K \supseteq \partial \text{conv} K$. Then there exist $cn \log 1/\epsilon $ Minkowski symmetrizations with respect to hyperplanes, that transform $K$ into a body $\tilde{K}$ that satisfies \[(1-\epsilon) w(K) B^n \subset \tilde{K} \subset (1+\epsilon) w(K) B^n,\] where $c>0$ is some numerical constant.
\end{theorem}
\begin{proof}
	First we consider the sequence given by the original statement of this theorem for the convex body $\text{conv}M_H K$.
	As we have proved in Theorem \ref{Tb}, applying the first symmetrization the resulting body will be $\text{conv}M_H K$. Proceeding with the sequence as in the original result, we conclude the proof.
\end{proof}

\section*{Acknowledgements}

The author would like to thank Gabriele Bianchi and Paolo Gronchi for the insightful help and the inputs for this work, which started as a Master Degree Thesis.

\end{document}